\documentclass[12pt]{amsart}
\usepackage{amssymb}
\usepackage{amsthm}
\usepackage{enumerate}
\usepackage[dvips]{graphics,color}

\theoremstyle{plain}
\newtheorem{Thm}{Theorem}[section]
\newtheorem{Lem}[Thm]{Lemma}

\theoremstyle{remark}

\small\normalsize
\numberwithin{equation}{section}

\newcommand{\C}{{\mathbb C}}

\newcommand{\R}{{\mathbb R}}
\newcommand{\D}{{\mathbb D}}

\begin{document}

\title{Integral means of the derivatives of Blaschke products}

\address{Universit\'e de Lyon; Universit\'e Lyon 1; Institut Camille Jordan CNRS UMR 5208; 43, boulevard du 11 Novembre 1918, F-69622 Villeurbanne.}
\email{fricain@math.univ-lyon1.fr}

\author{Emmanuel Fricain, Javad Mashreghi}
\address{D\'epartement de math\'ematiques et de statistique,
         Universit\'e Laval,
         Qu\'ebec, QC,
         Canada G1K 7P4.}
\email{Javad.Mashreghi@mat.ulaval.ca}

\thanks{This work was supported by NSERC (Canada) and FQRNT (Qu\'ebec). A part of this work was done while the first author was visiting McGill University. He would like to thank this institution for its warm hospitality.}



\begin{abstract}
We study the rate of growth of some integral means of the
derivatives of a Blaschke product and we generalize several
classical results. Moreover, we obtain the rate of growth of
integral means of the derivative of functions in the model subspace
$K_B$ generated by the Blaschke product $B$.
\end{abstract}

\maketitle

\section{Introduction}
Let $(z_n)_{n \geq 1}$ be a sequence in the unit disc satisfying the
Blaschke condition
\begin{equation} \label{E:blascon}
\sum_{n=1}^{\infty} (1-|z_n|) < \infty.
\end{equation}
Then, the product
\[
B(z) = \prod_{n=1}^{\infty} \frac{|z_n|}{z_n} \,\,
\frac{z_n-z}{1-\bar{z}_n \, z}
\]
is a bounded analytic function on the unit disc $\D$ with zeros only
at the points $z_n$, $n \geq 1$, \cite[page 20]{pD70}. Since the
product converges uniformly on compact subsets of $\D$, the
logarithmic derivative of $B$ is given by
\[
\frac{B'(z)}{B(z)} = \sum_{n=1}^{\infty}
\frac{1-|z_n|^2}{(1-\bar{z}_n \, z)(z-z_n)}, \hspace{1cm} (z \in
\D).
\]
Therefore,
\begin{equation} \label{E:estb'1}
|B'(re^{i\theta})| \leq \sum_{n=1}^{\infty} \frac{1-|z_n|^2}{|\,
1-\bar{z}_n \, re^{i\theta} \,|^2}, \hspace{1cm} (re^{i\theta} \in
\D).
\end{equation}
If (\ref{E:blascon}) is the only restriction we put on the zeros of
$B$, we can only say that
\begin{eqnarray*}
\int_{0}^{2\pi} |B'(re^{i\theta})| \, d\theta &\leq&
\sum_{n=1}^{\infty} (1-|z_n|^2) \, \int_{0}^{2\pi}
\frac{d\theta}{|\,
1-\bar{z}_n \, re^{i\theta} \,|^2}\\
&=& \sum_{n=1}^{\infty} (1-|z_n|^2) \,
\frac{2\pi}{(1-|z_n|^2r^2)}\\
&\leq&  \frac{4\pi\sum_{n=1}^{\infty} (1-|z_n|) }{(1-r)},
\end{eqnarray*}
which implies
\begin{equation} \label{E:norminfbp}
\int_{0}^{2\pi} |B'(re^{i\theta})| \, d\theta = \frac{o(1)}{1-r},
\hspace{1cm} (r \to 1).
\end{equation}
However, assuming stronger restrictions on the rate of increase of
the zeros of $B$ give us more precise estimates about the rate of
increase of integral means of $B_r'$ as $r \to 1$. The most common
restriction is
\begin{equation} \label{E:blasconal}
\sum_{n=1}^{\infty} (1-|z_n|)^\alpha < \infty
\end{equation}
for some $\alpha \in (0,1)$.  Protas \cite{Pr73} took the first step
in this direction by proving the following results.

Let us mention that $H^p$, $0<p<\infty$, stands for the classical
Hardy space equipped with the norm
\[
\|f\|_p = \lim_{r \to 1} \bigg(\, \int_{0}^{2\pi}
|f(re^{i\theta})|^p \,\,\, \frac{d\theta}{2\pi}
\,\bigg)^{\frac{1}{p}},
\]
and its cousin $A_\gamma^p$, $0<p<\infty$ and $\gamma > -1$, stands
for the (weighted) Bergman space equipped with the norm
\[
\|f\|_{p,\gamma} = \bigg(\,   \int_{0}^{1}\int_{0}^{2\pi}
|f(re^{i\theta})|^p \,\,\, \frac{r(1-r^2)^\gamma dr \,\,
d\theta}{\pi/(1+\gamma)} \,\bigg)^{\frac{1}{p}}.
\]

\begin{Thm}[Protas] \label{T:prot}
If $0 < \alpha < \frac{1}{2}$ and the Blaschke sequence $(z_n)_{n
\geq 1}$ satisfies (\ref{E:blasconal}), then $B' \in H^{1-\alpha}$.
\end{Thm}
\begin{Thm}[Protas] \label{T:prot2}
If $0 < \alpha < 1$ and the Blaschke sequence $(z_n)_{n \geq 1}$
satisfies (\ref{E:blasconal}), then $B' \in A_{\alpha-1}^{1}$.
\end{Thm}
\noindent Then, Ahern and Clark \cite{ak74} showed that Theorem
\ref{T:prot} is sharp in the sense that $B'$ need not lie in any
$H^p$ with $p > 1-\alpha$. Later on, they also showed that the
condition $\sum_{n=1}^{\infty} (1-|z_n|)^{1/2} < \infty$ is not
enough to imply that $B' \in H^{1/2}$ \cite{ak76}. At the same time,
Linden \cite{Li76} generalized Theorem \ref{T:prot} for higher
derivatives of $B$. In the converse direction, Ahern and Clark
\cite{ak74}  also obtained the following result.

\begin{Thm}[Ahern--Clark] \label{T:AhernClark}
If $\frac{1}{2} < p < 1$, then there is a Blaschke product $B$ with
$B' \in H^p$, and such that its zeros satisfies
\[
\sum_{n=1}^{\infty} (1-|z_n|)^\alpha = \infty
\]
for all $\alpha$ with $0 < \alpha < (1-p)$.
\end{Thm}
\noindent However, Cohn \cite{Co83} proved that for interpolating
sequences the two conditions are equivalent.

\begin{Thm}[Cohn] \label{T:cohn}
Let $0 < \alpha < \frac{1}{2}$, and let $(z_n)_{n \geq 1}$ be an
interpolating Blaschke sequence. Then, $B' \in H^{1-\alpha}$ if and
only if $(z_n)_{n \geq 1}$ satisfies (\ref{E:blasconal}).
\end{Thm}
\noindent Recently, Kutbi \cite{kutbi02} showed that under the
hypothesis of Theorem \ref{T:prot},
\begin{equation} \label{E:kutbi}
\int_{0}^{2\pi} |B'(re^{i\theta})|^p \, d\theta =
\frac{o(1)}{(1-r)^{p+\alpha-1}}, \hspace{1cm} (r \to 1),
\end{equation}
for any $p>1-\alpha$. In particular, for $p=1$, we have
\[
\int_{0}^{2\pi} |B'(re^{i\theta})| \, d\theta =
\frac{o(1)}{(1-r)^{\alpha}}, \hspace{1cm} (r \to 1),
\]
which is a refinement of (\ref{E:norminfbp}).

Then, Protas \cite{Pr04} proved that the estimate (\ref{E:kutbi}) is still valid if $1/2<\alpha\leq 1$, $p\geq \alpha$ and the Blaschke sequence $(z_n)_{n\geq 1}$ satisfies (\ref{E:blasconal}). Finally, Gotoh \cite{Go07} got an extension of Protas's results for higher derivatives of $B$.

A Blaschke sequence which satisfies the Carleson condition is called
an interpolation, or Carleson, Blaschke sequence \cite[page
200]{pK98}. Let $I$ be an inner function for the unit disc. In
particular, $I$ could be any Blaschke product. Then,
\[
K_I :=H^2 \ominus I H^2
\]
is called the model subspace of $H^2$ generated by the inner
function $I$ \cite{effun, hmI-II}.  Cohn \cite{Co83} obtained the
following result about the derivative of functions in $K_B$.
\begin{Thm}[Cohn] \label{T:cohn2}
Let $(z_n)_{n \geq 1}$ be an interpolating Blaschke sequence, and
let $p \in (2/3,1)$. Then, $B' \in H^p$ if and only if $f' \in
H^{2p/(p+2)}$ for all $f \in K_B$.
\end{Thm}

In this paper, we replace the condition (\ref{E:blasconal}) by a
more general assumption
\begin{equation} \label{E:blasconalgen}
\sum_{n=1}^{\infty} h(1-|z_n|) < \infty,
\end{equation}
where $h$ is a positive continuous function satisfying certain
smoothness conditions, and then we generalize all the preceding
results. Since our sequence already satisfies the Blaschke
condition, (\ref{E:blasconalgen}) will provide further information
about the rate of increase of the zeros only if $h(t) \geq t$ as $t
\to 0$.

In particular, we are interested in
\begin{equation} \label{E:defh}
h(t) = t^\alpha \,\, (\log 1/t)^{\alpha_1} \,\, (\log_2
1/t)^{\alpha_2} \,\, \cdots \,\, \,\, (\log_n 1/t)^{\alpha_n},
\end{equation}
where $\alpha \in (0,1)$, $\alpha_1,\alpha_2, \cdots, \alpha_n \in
\R$, and $\log_n = \log\log\cdots\log$ ($n$ times) \cite{jmcmft}.

In the following, we will use the estimates
\begin{eqnarray*}
\int_{0}^{2\pi} \frac{d\theta}{|1-re^{i\theta}|^\nu} &\asymp&
\frac{1}{(1-r)^{\nu-1}}, \hspace{1cm} (\nu>1),\\
\int_{0}^{1}\int_{0}^{2\pi} \frac{(1-\rho^2)^\gamma}{|1-r \rho
e^{i\theta}|^\nu} \,\,  \, \rho d\rho d\theta &\asymp&
\frac{1}{(1-r)^{\nu-\gamma-2}}, \hspace{1cm} (\nu-2>\gamma>-1),
\end{eqnarray*}
as $r \to 1^-$. See \cite[page 7]{hkz}. Both relations can be proved
using the fact that $|1-re^{i\theta}| \asymp (1-r) + |\theta|$ as $r
\to 1^-$.

\section{An estimation lemma}
In the following we assume that $h$ is a continuous positive
function defined on the interval $(0,1)$ with
\[
\lim_{t \to 0^+} h(t)=0.
\]
Our prototype is the one given in (\ref{E:defh}). The following
lemma has simple assumptions and also a very simple proof. However,
it has  many interesting applications in the rest of the paper.

\begin{Lem}  \label{L:estim}
Let $(r_n)_{n \geq 1}$ be a sequence in the interval $(0,1)$ such
that
\[
\sum_{n=1}^{\infty} h(1-r_n) < \infty.
\]
Let $p,q>0$ be such that $h(t)/t^p$ is decreasing and $h(t)/t^{p-q}$
is increasing on $(0,1)$. Then,
\[
\sum_{n=1}^{\infty} \frac{(1-r_n)^p}{(1-rr_n)^q} =
\frac{O(1)}{(1-r)^{q-p} \, h(1-r)}
\]
as $r \to 1^-$. Moreover, if
\[
\lim_{t \to 0^+} \frac{h(t)}{t^{p-q}} = 0,
\]
then
\[
\sum_{n=1}^{\infty} \frac{(1-r_n)^p}{(1-rr_n)^q} =
\frac{o(1)}{(1-r)^{q-p} \, h(1-r)}.
\]
\end{Lem}

\begin{proof}
We have
\[
\frac{(1-r_n)^p}{(1-rr_n)^q} =
\bigg(\, \frac{(1-r_n)^p}{h(1-r_n)} \, \frac{h(1-rr_n)}{(1-rr_n)^p}
\,\bigg) \,\,
\bigg(\,  \frac{h(1-r_n)}{(1-rr_n)^{q-p} \, h(1-rr_n)} \,\bigg).
\]
By assumption
\[
\frac{h(1-rr_n)}{(1-rr_n)^p} \leq \frac{h(1-r_n)}{(1-r_n)^p},
\]
and
\[
(1-rr_n)^{q-p} \, h(1-rr_n) \geq (1-r)^{q-p} \, h(1-r).
\]
Thus, for any $n \geq 1$,
\begin{equation} \label{E:estim1}
\frac{(1-r_n)^p}{(1-rr_n)^q} \leq
\frac{h(1-r_n)}{(1-r)^{q-p} \, h(1-r)}.
\end{equation}
Given $\varepsilon>0$, fix $N$ such that
\[
\sum_{n=N+1}^{\infty} h(1-r_n) < \varepsilon.
\]
Hence, by (\ref{E:estim1}),
\begin{eqnarray*}
\sum_{n=1}^{\infty} \frac{(1-r_n)^p}{(1-rr_n)^q} &=& \sum_{n=1}^{N}
\frac{(1-r_n)^p}{(1-rr_n)^q} +
\sum_{n=N+1}^{\infty} \frac{(1-r_n)^p}{(1-rr_n)^q}\\
&\leq& \sum_{n=1}^{N} (1-r_n)^{p-q} + \frac{\sum_{n=N+1}^{\infty}
h(1-r_n)}{(1-r)^{q-p} \, h(1-r)}\\
&\leq& C_N + \frac{\varepsilon}{(1-r)^{q-p} \, h(1-r)},
\end{eqnarray*}
where $C_N$ is independent of $r$. This inequality implies both assertions of the Lemma.
\end{proof}

\noindent The Lemma is still valid if instead of ``decreasing" and
``increasing", we assume that our functions are respectively
``boundedly decreasing" and ``boundedly increasing". We say that
$\varphi$ is boundedly increasing if there is a constant $C>0$ such
that $\varphi(x) \leq C \varphi(y)$ whenever $x \leq y$. Similarly,
$\varphi$ is boundedly decreasing if there is a constant $C>0$ such
that $\varphi(x) \geq C \varphi(y)$ whenever $x \leq y$.

\section{$H^p$ means of the first derivative}
In this section we apply Lemma \ref{L:estim} to obtain a general
estimate for the integral means of the first derivative of a
Blaschke product. Special cases of the following theorem generalize
Protas and Kutbi's results.

\begin{Thm} \label{T:caseles12}
Let $B$ be the Blaschke product formed with zeros $z_n=r_n
e^{i\theta_n}$, $n \geq 1$, satisfying
\[
\sum_{n=1}^{\infty} h(1-r_n) < \infty
\]
for a positive continuous function $h$. Suppose that there is $q \in
(1/2,1]$ such that $h(t)/t^q$ is decreasing and $h(t)/t^{1-q}$ is
increasing on $(0,1)$. Then, for any $p \geq q$,
\[
\int_{0}^{2\pi} |B'(re^{i\theta})|^p \, d\theta =
\frac{O(1)}{(1-r)^{p-1} \, h(1-r)}, \hspace{1cm} (r \to 1).
\]
Moreover, if $\lim_{t \to 0} h(t)/t^{1-q} = 0$, then $O(1)$ can be
replaced by $o(1)$.
\end{Thm}

\begin{proof}
Since $q \leq 1$, (\ref{E:estb'1}) implies
\[
|B'(re^{i\theta})|^q \leq \sum_{n=1}^{\infty} \frac{(1-r_n^2)^q}{|\,
1-rr_ne^{i(\theta-\theta_n)} \,|^{2q}}.
\]
Hence
\begin{equation} \label{E:firstder}
\int_{0}^{2\pi} |B'(re^{i\theta})|^q \, d\theta  \leq C \,
\sum_{n=1}^{\infty} \frac{(1-r_n)^q}{(1-rr_n)^{2q-1}}.
\end{equation}
(Here we used $2q>1$.) Therefore, by Lemma \ref{L:estim},
\[
\int_{0}^{2\pi} |B'(re^{i\theta})|^q \, d\theta  \leq
 \frac{C}{(1-r)^{q-1} \, h(1-r)}.
\]
Now recall that any function $f$ in $H^\infty$  is in the Bloch space $\mathcal B$ \cite[page 44]{pd04}, that is 
\[
\sup_{z\in\mathbb D}(1-|z|^2)|f'(z)|<+\infty\,.
\]
Hence, for any $p \geq q$,
\[
\int_{0}^{2\pi} |B'(re^{i\theta})|^p \, d\theta  \leq
 \frac{1}{(1-r)^{p-q}} \, \int_{0}^{2\pi} |B'(re^{i\theta})|^q \,\, d\theta
 \leq \frac{C}{(1-r)^{p-1} \, h(1-r)}.
\]
Finally, as $r \to 1$, Lemma \ref{L:estim} also assures that $C$ can
be replaced by any small positive constant if $\lim_{t \to 0}
h(t)/t^{1-q} = 0$.
\end{proof}

\noindent Now, we can apply Theorem \ref{T:caseles12} for the
special function defined in (\ref{E:defh}).

{\bf Case I:}  If
\[
\sum_{n=1}^\infty (1-r_n)^{\alpha} (\log \frac{1}{1-r_n}
)^{\alpha_1}\cdots (\log_m \frac{1}{1-r_n} )^{\alpha_m} <\infty,
\]
then, for any
\[
p > \max\{ \alpha,1-\alpha \}
\]
we have
\[
\int_{0}^{2\pi} |B'(re^{i\theta})|^p \, d\theta = \frac{o(1)}{
(1-r)^{\alpha+p-1} (\log
 \frac{1}{1-r} )^{\alpha_1}\cdots (\log_m \frac{1}{1-r} )^{\alpha_m} },
\hspace{1cm} (r \to 1).
\]

\noindent In particular, if
\[
\sum_{n=1}^\infty (1-r_n)^{\alpha}  <\infty,
\]
 with $\alpha \in (0,1/2)$, then, for
any $p > 1-\alpha$,
\[
\int_{0}^{2\pi} |B'(re^{i\theta})|^p \, d\theta =
\frac{o(1)}{(1-r)^{p+\alpha-1}}, \hspace{1cm} (r \to 1),
\]
which is Kutbi's result. Moreover, if $\alpha \in [1/2,1)$, the last
estimate still holds for any $p > \alpha$, which is Protas's result \cite{Pr04}.

{\bf Case II:} If
\[
\sum_{n=1}^\infty (1-r_n)^{\alpha} (\log_k \frac{1}{1-r_n}
)^{\alpha_k}\cdots (\log_m \frac{1}{1-r_n} )^{\alpha_m} <\infty,
\]
with $\alpha \in (0,1/2)$, $\alpha_k<0$ and $\alpha_{k+1}, \cdots,
\alpha_m \in \R$, then,
\[
\int_{0}^{2\pi} |B'(re^{i\theta})|^{1-\alpha} \, d\theta =
\frac{o(1)}{ (\log_k
 \frac{1}{1-r} )^{\alpha_k}\cdots (\log_m \frac{1}{1-r} )^{\alpha_m} }, \hspace{1cm} (r \to 1).
\]

\noindent But, if
\[
\sum_{n=1}^\infty (1-r_n)^{\alpha}  <\infty,
\]
with $\alpha \in (0,1/2)$, then
\[
\int_{0}^{2\pi} |B'(re^{i\theta})|^{1-\alpha} \, d\theta = O(1),
\hspace{1cm} (r \to 1),
\]
i.e. $B' \in H^{1-\alpha}$, which is Protas' result \cite{Pr73}.

{\bf Case III:} If
\[
\sum_{n=1}^\infty (1-r_n)^{\alpha} (\log_k \frac{1}{1-r_n}
)^{\alpha_k}\cdots (\log_m \frac{1}{1-r_n} )^{\alpha_m} <\infty,
\]
with $\alpha \in (1/2,1)$, $\alpha_k>0$ and $\alpha_{k+1}, \cdots,
\alpha_n \in \R$, then,
\[
\int_{0}^{2\pi} |B'(re^{i\theta})|^{\alpha} \, d\theta =
\frac{o(1)}{ (1-r)^{2\alpha-1} (\log
 \frac{1}{1-r} )^{\alpha_1}\cdots (\log_m \frac{1}{1-r} )^{\alpha_m} }, \hspace{1cm} (r \to 1).
\]

\noindent However, if
\[
\sum_{n=1}^\infty (1-r_n)^{\alpha}  <\infty,
\]
with $\alpha \in (1/2,1)$, then we still have
\[
\int_{0}^{2\pi} |B'(re^{i\theta})|^{\alpha} \, d\theta =
\frac{o(1)}{(1-r)^{2\alpha-1}}, \hspace{1cm} (r \to 1).
\]

\section{$H^p$  means of higher derivatives}
Straightforward calculation leads to
\begin{equation} \label{E:higherder}
\int_{0}^{2\pi} |B^{(\ell)} (re^{i\theta})|^p \, d\theta  \leq
C(p,\ell) \, \, \sum_{n=1}^{\infty}
\frac{(1-r_n)^p}{(1-rr_n)^{(\ell+1)p-1}}, \hspace{1cm}
(\frac{1}{\ell+1} < p \leq \frac{1}{\ell}),
\end{equation}
which is a generalization of (\ref{E:firstder}). This observation
along with Lemma \ref{L:estim} enable us to generalize the results
of the preceding section for higher derivatives of a Blaschke
product. The proof is similar to that of Theorem \ref{T:caseles12}.

\begin{Thm} \label{T:caseles12gen}
Let $B$ be the Blaschke product formed with zeros $z_n=r_n
e^{i\theta_n}$, $n \geq 1$, satisfying
\[
\sum_{n=1}^{\infty} h(1-r_n) < \infty
\]
for a positive continuous function $h$. Suppose that there is $q \in
(1/(\ell+1),1/\ell]$ such that $h(t)/t^q$ is decreasing and
$h(t)/t^{1-\ell q}$ is increasing on $(0,1)$. Then, for any $p \geq
q$,
\[
\int_{0}^{2\pi} |B^{(\ell)} (re^{i\theta})|^p \, d\theta =
\frac{O(1)}{(1-r)^{\ell p-1} \, h(1-r)}, \hspace{1cm} (r \to 1).
\]
Moreover, if $\lim_{t \to 0} h(t)/t^{1-\ell q} = 0$, then $O(1)$ can
be replaced by $o(1)$.
\end{Thm}

\noindent Now, we can apply Theorem \ref{T:caseles12gen} for the
special function defined in (\ref{E:defh}).

{\bf Case I:}  If
\[
\sum_{n=1}^\infty (1-r_n)^{\alpha} (\log \frac{1}{1-r_n}
)^{\alpha_1}\cdots (\log_m \frac{1}{1-r_n} )^{\alpha_m} <\infty,
\]
then, for any
\[
p > \max\{ \alpha,(1-\alpha)/\ell \}
\]
we have
\[
\int_{0}^{2\pi} |B^{(\ell)} (re^{i\theta})|^p \, d\theta =
\frac{o(1)}{ (1-r)^{\alpha+\ell p-1} (\log
 \frac{1}{1-r} )^{\alpha_1}\cdots (\log_m \frac{1}{1-r} )^{\alpha_m} }, \hspace{1cm} (r \to 1).
\]

\noindent In particular, if
\[
\sum_{n=1}^\infty (1-r_n)^{\alpha}  <\infty,
\]
with $\alpha \in (0,1/(\ell+1))$, then, for any $p >
(1-\alpha)/\ell$,
\[
\int_{0}^{2\pi} |B^{(\ell)}(re^{i\theta})|^p \, d\theta =
\frac{o(1)}{(1-r)^{\ell p+\alpha-1}}, \hspace{1cm} (r \to 1),
\]
which is Kutbi's result. Moreover, if $\alpha \in [1/(\ell+1),1)$,
the last estimate still holds for any $p > \alpha$, which is Gotoh's result \cite{Go07} .

{\bf Case II:} If
\[
\sum_{n=1}^\infty (1-r_n)^{\alpha} (\log_k \frac{1}{1-r_n}
)^{\alpha_k}\cdots (\log_m \frac{1}{1-r_n} )^{\alpha_m} <\infty,
\]
with $\alpha \in (0,1/(1+\ell))$, $\alpha_k<0$ and $\alpha_{k+1},
\cdots, \alpha_n \in \R$, then,
\[
\int_{0}^{2\pi} |B^{(\ell)} (re^{i\theta})|^{(1-\alpha)/\ell} \,\,
d\theta = \frac{o(1)}{ (\log_k
 \frac{1}{1-r} )^{\alpha_k}\cdots (\log_m \frac{1}{1-r} )^{\alpha_m} }, \hspace{1cm} (r \to 1).
\]

\noindent But, if
\[
\sum_{n=1}^\infty (1-r_n)^{\alpha}<+\infty,
\]
with $\alpha \in (0,1/(1+\ell))$, then
\[
\int_{0}^{2\pi} |B^{(\ell)}(re^{i\theta})|^{(1-\alpha)/\ell} \,
d\theta = O(1), \hspace{1cm} (r \to 1),
\]
i.e. $B^{(\ell)} \in H^{(1-\alpha)/\ell}$ which is Linden's result
\cite{Li76}.

{\bf Case III:}  If
\[
\sum_{n=1}^\infty (1-r_n)^{\alpha} (\log_k \frac{1}{1-r_n}
)^{\alpha_k}\cdots (\log_m \frac{1}{1-r_n} )^{\alpha_m} <\infty,
\]
with $\alpha \in (1/(1+\ell),1)$, $\alpha_k>0$ and $\alpha_{k+1},
\cdots, \alpha_n \in \R$, then,
\[
\int_{0}^{2\pi} |B^{(\ell)} (re^{i\theta})|^{\alpha} \,\, d\theta =
\frac{o(1)}{ (1-r)^{(\ell+1) \alpha-1}  (\log_k
 \frac{1}{1-r} )^{\alpha_k}\cdots (\log_m \frac{1}{1-r} )^{\alpha_m} }, \hspace{1cm} (r \to 1).
\]

\noindent However, if
\[
\sum_{n=1}^\infty (1-r_n)^{\alpha}  <\infty,
\]
with $\alpha \in (1/(\ell+1),1)$, then we still have
\[
\int_{0}^{2\pi} |B^{(\ell)} (re^{i\theta})|^\alpha \, d\theta =
\frac{o(1)}{(1-r)^{(\ell+1)\alpha-1}}, \hspace{1cm} (r \to 1).
\]

\section{$A^p_\gamma$  means of the first derivative}
In this section we apply Lemma \ref{L:estim} to obtain a general
estimate for the integral means of the first derivative of a
Blaschke product. Special cases of the following theorem generalize
Protas's results \cite{Pr73}.

\begin{Thm} \label{T:caseles12-berg}
Let $B$ be the Blaschke product formed with zeros $z_n=r_n
e^{i\theta_n}$ satisfying
\[
\sum_{n=1}^{\infty} h(1-r_n) < \infty
\]
for a positive continuous function $h$. Let $\gamma \in (-1,0)$.
Suppose that there is $q \in (1+\gamma/2,1]$ such that $h(t)/t^q$ is
decreasing and $h(t)/t^{2+\gamma-q}$ is increasing on $(0,1)$. Then,
for any $p \geq q$,
\[
\int_{0}^{1}\int_{0}^{2\pi} |B'(r\rho e^{i\theta})|^p \,\,\,
\rho(1-\rho^2)^\gamma d\rho \,\, d\theta =
\frac{O(1)}{(1-r)^{p-\gamma-2} \, h(1-r)}, \hspace{1cm} (r \to 1).
\]
Moreover, if $\lim_{t \to 0} h(t)/t^{2+\gamma-q} = 0$, then $O(1)$
can be replaced by $o(1)$.
\end{Thm}

\begin{proof}We saw that
\[
|B'(r \rho e^{i\theta})|^q \leq \sum_{n=1}^{\infty}
\frac{(1-r_n^2)^q}{|\, 1-r  r_n \rho e^{i(\theta-\theta_n)}
\,|^{2q}}.
\]
Hence
\begin{equation} \label{E:firstderberg}
\int_{0}^{1}\int_{0}^{2\pi} |B'(r \rho e^{i\theta})|^q \,\,\,
\rho(1-\rho^2)^\gamma d\rho \,\, d\theta  \leq C \,
\sum_{n=1}^{\infty} \frac{(1-r_n)^q}{(1-rr_n)^{2q-\gamma-2}}.
\end{equation}
(Here we used $2q-\gamma-2>0$.) Therefore, by Lemma \ref{L:estim},
\[
\int_{0}^{1}\int_{0}^{2\pi} |B'(r \rho e^{i\theta})|^q \,\,\,
\rho(1-\rho^2)^\gamma d\rho \,\, d\theta  \leq
 \frac{C}{(1-r)^{q-\gamma-2} \, h(1-r)}.
\]
Hence, for any $p \geq q$,
\begin{eqnarray*}
\int_{0}^{1}\int_{0}^{2\pi} |B'(r \rho e^{i\theta})|^p \,\,\,
\rho(1-\rho^2)^\gamma d\rho \,\, d\theta
&\leq&
 \frac{1}{(1-r)^{p-q}} \, \int_{0}^{1}\int_{0}^{2\pi} |B'(r \rho e^{i\theta})|^q \,\,\,
\rho(1-\rho^2)^\gamma d\rho \,\, d\theta\\
&\leq& \frac{C}{(1-r)^{p-\gamma-2} \, h(1-r)}.
\end{eqnarray*}
Finally, as $r \to 1$, Lemma \ref{L:estim} also assures that $C$ can
be replaced by any small positive constant if $\lim_{t \to 0}
h(t)/t^{2+\gamma-q} = 0$.
\end{proof}

\noindent Now, we can apply Theorem \ref{T:caseles12-berg} for the
special function defined in (\ref{E:defh}).

{\bf Case I:} If
\[
\sum_{n=1}^\infty (1-r_n)^{\alpha} (\log \frac{1}{1-r_n}
)^{\alpha_1}\cdots (\log_m \frac{1}{1-r_n} )^{\alpha_m} <\infty,
\]
and if $\gamma \in (-1, \alpha -1)$, then, for any
\[
p > \max\{\, \alpha,\, 2+\gamma-\alpha, \, 1+\gamma/2 \,\},
\]
we have
\[
\int_{0}^{1}\int_{0}^{2\pi} |B'(r\rho e^{i\theta})|^p \,\,\,
\rho(1-\rho^2)^\gamma d\rho \,\, d\theta =
\frac{o(1)}{(1-r)^{\alpha+p-\gamma-2} \, (\log
 \frac{1}{1-r} )^{\alpha_1} \cdots (\log_m \frac{1}{1-r} )^{\alpha_m} },
\]
as $(r \to 1)$. In particular, if
\[
\sum_{n=1}^\infty (1-r_n)^{\alpha}  <\infty,
\]
then
\[
\int_{0}^{1}\int_{0}^{2\pi} |B'(r\rho e^{i\theta})|^p \,\,\,
\rho(1-\rho^2)^\gamma d\rho \,\, d\theta =
\frac{o(1)}{(1-r)^{p+\alpha-\gamma-2}}.
\]

{\bf Case II:}
If
\[
\sum_{n=1}^\infty (1-r_n)^{\alpha} (\log_k \frac{1}{1-r_n}
)^{\alpha_k}\cdots (\log_m \frac{1}{1-r_n} )^{\alpha_m} <\infty,
\]
with  $\alpha_k < 0$, then, for any $p \geq 1$,
\[
\int_{0}^{1}\int_{0}^{2\pi} |B'(r\rho e^{i\theta})|^p \,\,\,
\rho(1-\rho^2)^{\alpha-1} d\rho \,\, d\theta = \frac{o(1)}{
(1-r)^{p-1} \, (\log_k \frac{1}{1-r} )^{\alpha_k} \cdots (\log_m
\frac{1}{1-r} )^{\alpha_m}}, \hspace{1cm} (r \to 1).
\]

{\bf Case III:} If
\[
\sum_{n=1}^\infty (1-r_n)^{\alpha} < \infty,
\]
then, for any $p \geq 1$,
\[
\int_{0}^{1}\int_{0}^{2\pi} |B'(r\rho e^{i\theta})|^p \,\,\,
\rho(1-\rho^2)^{\alpha-1} d\rho \,\, d\theta =
\frac{O(1)}{(1-r)^{p-1}}, \hspace{1cm} (r \to 1).
\]
In particular, for $p=1$,
\[
\int_{0}^{1}\int_{0}^{2\pi} |B'(r\rho e^{i\theta})| \,\,\,
\rho(1-\rho^2)^{\alpha-1} d\rho \,\, d\theta = O(1), \hspace{1cm} (r
\to 1).
\]
which is the Protas' result \cite{Pr73}.

Some other cases can also be considered here. But, since they are
immediate consequence of Theorem \ref{T:caseles12-berg}, we do not
proceed further. Moreover, using similar techniques, one can obtain
estimates for the $A^p_\gamma$  means of the higher derivatives for
a Blaschke product satisfying the hypothesis of Theorem
\ref{T:caseles12-berg}.

\section{Interpolating Blaschke products}
Cohn's theorems \ref{T:cohn} and \ref{T:cohn2} imply that if
$z_n=r_n e^{i\theta_n}$, $n\geq 1$, is a Carleson sequence
satisfying
\[
\sum_{n=1}^\infty (1-r_n)^{1-p} <\infty
\]
for some $p \in (2/3,1)$, then $f' \in H^{2p/(p+2)}$ for all $f \in
K_B$. The following result generalizes this fact.

\begin{Thm} \label{T:modhard}
Let $z_n=r_n e^{i\theta_n}$, $n\geq 1$, be a Carleson sequence
satisfying
\[
\sum_{n=1}^\infty h(1-r_n)<\infty
\]
for a positive continuous function $h$. Let $B$ be the Blaschke
product formed with zeros $z_n$, $n\geq 1$. Suppose that there is
$p\in (2/3,1)$ such that $h(t)/t^{p/2}$ is decreasing and
$h(t)/{t^{1-p}}$ is increasing on $(0,1)$. Then, for all $f \in
K_B$, we have
\[
\left(\int_0^{2\pi}\left|f'(re^{i\theta})\right|^\sigma\,d\theta\right)^{1/\sigma}
\leq \frac {C \, \|f\|_2}{(\, (1-r)^{p-1}h(1-r) \,)^{1/p}},\qquad
(r\to 1),
\]
with $\sigma=2p/(p+2)$ and $C$ an absolute constant.
\end{Thm}

\begin{proof}
Since $(z_n)_{n\geq 1}$ is a Carleson sequence, we know that the
functions
\[
f_n(z):=\frac {(1-r_n)^{1/2}}{1-\overline{z}_n \, z},\qquad (n \geq
1),
\]
form a Riesz basis of $K_B$ (see \cite{treatise} for instance). Now, let $f=\displaystyle\sum_{n=1}^N
\beta_n f_n$, $\beta_n \in \C$. Then
\[
f'(z)=\sum_{n=1}^N \frac{\overline{z}_n \, \beta_n
(1-r_n)^{1/2}}{(1-\overline{z}_n \, z)^2},
\]
 and thus we get
\[
|f'(z)|\leq \sum_{n=1}^N
\frac{|\beta_n|(1-r_n)^{1/2}}{|1-\overline{z}_n \, z|^2}.
\]
Since $p\in (2/3,1)$, we have $\sigma\in (1/2,1)$ and we can write
\[
|f'(z)|^\sigma\leq \sum_{n=1}^N
\frac{|\beta_n|^\sigma(1-r_n)^{\sigma/2}}{|1-\overline{z}_n \,
z|^{2\sigma}}.
\]
Therefore,
\begin{eqnarray*}
\int_{0}^{2\pi} \left|f'(re^{i\theta})\right|^\sigma\,d\theta &\leq&
\sum_{n=1}^N |\beta_n|^\sigma (1-r_n)^{\sigma/2}\int_0^{2\pi}\frac
{d\theta}{|1-\overline{z}_n re^{i\theta}|^{2\sigma}}\\
&\leq& c\sum_{n=1}^N |\beta_n|^\sigma \frac
{(1-r_n)^{\sigma/2}}{(1-rr_n)^{2\sigma-1}}.
\end{eqnarray*}
Let $p'=2/\sigma$ and let $q'$ be its conjugate exponent. Then
H\"older's inequality implies that
\[
\int_0^{2\pi}\left|f'(re^{i\theta})\right|^\sigma\,d\theta\leq
c\left(\sum_{n=1}^N |\beta_n|^2 \right)^{1/p'} \left(\sum_{n=1}^N
\frac {(1-r_n)^{\sigma
q'/2}}{(1-rr_n)^{(2\sigma-1)q'}}\right)^{1/q'}.
\]
But since $(f_n)_{n\geq 1}$ forms a Riesz basis of $K_B$, there
exists a constant $c_1>0$ such that
\[
\sum_{n=1}^N |\beta_n|^2 \leq c_1 \|f\|_2^2,
\]
 whence
\[
\int_0^{2\pi}\left|f'(re^{i\theta})\right|^\sigma\,d\theta\leq c_2
\|f\|_2^\sigma  \left(\sum_{n=1}^N \frac {(1-r_n)^{\sigma
q'/2}}{(1-rr_n)^{(2\sigma-1)q'}}\right)^{1/q'}.
\]
Now easy computations show that $q'=\frac {p+2}2$, $\sigma q'=p$,
$(2\sigma-1)q'=3p/2 -1$ and therefore, by Lemma 2.1, we have
\[
\sum_{n=1}^N\frac {(1-r_n)^{\sigma
q'/2}}{(1-rr_n)^{(2\sigma-1)q'}}\leq \frac C{(1-r)^{p-1}h(1-r)},
\]
where $C$ is a constant independent of $N$. We deduce that
\[
\int_0^{2\pi}\left|f'(re^{i\theta})\right|^\sigma\,d\theta\leq \frac
{c_3 \|f\|_2^\sigma} {((1-r)^{p-1}h(1-r))^{1/q'}}.
\]
Since $1/{\sigma q'}=1/p$, and using a density argument, we get that
for all $f\in K_B$,
\[
\left(\int_0^{2\pi}\left|f'(re^{i\theta})\right|^\sigma\,d\theta\right)^{1/\sigma}\leq
\frac {c_3^{1/\sigma} \|f\|_2}{((1-r)^{p-1}h(1-r))^{1/p}}.
\]
\end{proof}

\noindent Now, we can apply Theorem \ref{T:modhard} for the special
function defined in (\ref{E:defh}).

{\bf Case I:} If $z_n=r_n e^{i\theta_n}$, $n\geq 1$, is a Carleson
sequence satisfying
\[
\sum_{n=1}^\infty (1-r_n)^\alpha (\log \frac{1}{1-r_n}
)^{\alpha_1}\cdots (\log_m \frac{1}{1-r_n} )^{\alpha_m} <\infty,
\]
with  $p\in (2/3,1)$, $1-p<\alpha <p/2$, and
$\alpha_1,\dots,\alpha_m \in \R$, then, for all $f\in K_B$, we have
\[
\left(\int_0^{2\pi}\left|f'(re^{i\theta})\right|^\sigma\,d\theta\right)^{1/\sigma}
\leq \frac {C \, \|f\|_2}{\big(\, (1-r)^{\alpha+p-1} (\log
 \frac{1}{1-r} )^{\alpha_1}\cdots (\log_m \frac{1}{1-r} )^{\alpha_m}
\,\big)^{1/p}},
\]
with $\sigma=2p/(p+2)$ and $C$ an absolute constant.

{\bf Case II:} If $z_n=r_n e^{i\theta_n}$, $n\geq 1$, is a Carleson
sequence satisfying
\[
\sum_{n=1}^\infty (1-r_n)^{1-p} (\log_k \frac{1}{1-r_n}
)^{\alpha_k}\cdots (\log_m \frac{1}{1-r_n} )^{\alpha_m} <\infty,
\]
with $p\in (2/3,1)$, $k\geq 1$,
$\alpha_k,\alpha_{k+1},\dots,\alpha_m \in \R$, and $\alpha_k < 0$,
then, for all $f\in K_B$, we have
\[
\left(\int_0^{2\pi}\left|f'(re^{i\theta})\right|^\sigma\,d\theta\right)^{1/\sigma}
\leq \frac {C \, \|f\|_2}{\big(\, (\log_k \frac{1}{1-r}
)^{\alpha_k}\cdots (\log_m \frac{1}{1-r} )^{\alpha_m}
\,\big)^{1/p}},\qquad (r\to 1),
\]
with $\sigma=2p/(p+2)$ and $C$ an absolute constant. However, if
\[
\sum_{n=1}^\infty (1-r_n)^{1-p} <\infty,
\]
then we still have
\[
\left(\int_0^{2\pi}\left|f'(re^{i\theta})\right|^\sigma\,d\theta\right)^{1/\sigma}
\leq C \, \|f\|_2, \qquad (r\to 1),
\]
i.e. $f'\in H^{2p/(p+2)}$, for any $f\in K_B$, which is Cohn's
result.

{\bf Case III:} If $z_n=r_n e^{i\theta_n}$, $n\geq 1$, is a Carleson
sequence satisfying
\[
\sum_{n=1}^\infty (1-r_n)^{p/2} (\log_k \frac{1}{1-r_n}
)^{\alpha_k}\cdots (\log_m \frac{1}{1-r_n} )^{\alpha_m} <\infty,
\]
with $p\in (2/3,1)$, $k\geq 1$,
$\alpha_k,\alpha_{k+1},\dots,\alpha_m \in \R$, and $\alpha_k > 0$,
then, for all $f\in K_B$, we have
\[
\left(\int_0^{2\pi}\left|f'(re^{i\theta})\right|^\sigma\,d\theta\right)^{1/\sigma}
\leq \frac {C \, \|f\|_2}{\big(\, (1-r)^{3p/2-1}\, (\log_k
\frac{1}{1-r} )^{\alpha_k}\cdots (\log_m \frac{1}{1-r} )^{\alpha_m}
\,\big)^{1/p}},
\]
with $\sigma=2p/(p+2)$ and $C$ an absolute constant. However, if
\[
\sum_{n=1}^\infty (1-r_n)^{p/2}  <\infty,
\]
then we still have
\[
\left(\int_0^{2\pi}\left|f'(re^{i\theta})\right|^\sigma\,d\theta\right)^{1/\sigma}
\leq \frac {C \, \|f\|_2}{(1-r)^{(3p-2)/2p} }.
\]

Using similar techniques we can obtain some results about the
$A_\gamma^p$ means of the derivatives of function in the model
subspaces of $H^2$.

\begin{Thm} \label{T:modberg}
Let $z_n=r_n e^{i\theta_n}$, $n\geq 1$, be a Carleson sequence
satisfying
\[
\sum_{n=1}^\infty h(1-r_n)<\infty
\]
for a positive continuous function $h$, and let $B$ be the Blaschke
product formed with zeros $z_n$, $n\geq 1$. Let $p\in (2/3,1)$,
$\sigma=2p/(p+2)$ and $-1 < \gamma < 2\sigma-2$ such that
$h(t)/t^{p/2}$ is decreasing and
$h(t)/{t^{(1-p)+(1+\gamma)(1+p/2)}}$ is increasing on $(0,1)$. Then,
for all $f \in K_B$, we have
\[
\left( \int_{0}^{1}\int_{0}^{2\pi} \, \left|f'(r \, \rho
e^{i\theta})\right|^\sigma\, \, \rho(1-\rho^2)^\gamma d\rho \,
d\theta \right)^{1/\sigma}
\leq
\frac {C \, \|f\|_2}{\big(\, (1-r)^{-(1-p)-(1+\gamma)(1+p/2)} \,
h(1-r) \,\big)^{1/p}}
\]
as $r\to 1^-$.
\end{Thm}

\begin{proof}
The beginning of the proof is as of Theorem \ref{T:modhard} until
\[
|f'(z)|^\sigma\leq \sum_{n=1}^N
\frac{|\beta_n|^\sigma(1-r_n)^{\sigma/2}}{|1-\overline{z}_n \,
z|^{2\sigma}}.
\]
Therefore, by H\"older's inequality (with $p'=2/\sigma$ and $q'$ its
conjugate exponent) and by Lemma 2.1, we have
\begin{eqnarray*}
\int_{0}^{1}\int_{0}^{2\pi} \, \left|f'(r \, \rho
e^{i\theta})\right|^\sigma\, \, \rho(1-\rho^2)^\gamma d\rho \,
d\theta &\leq& \sum_{n=1}^N |\beta_n|^\sigma (1-r_n)^{\sigma/2}
\int_{0}^{1}\int_{0}^{2\pi}
\frac{\rho(1-\rho^2)^\gamma d\rho \,d\theta}{|1-\overline{z}_n r \rho e^{i\theta}|^{2\sigma}}\\
&\leq& c\sum_{n=1}^N |\beta_n|^\sigma \frac
{(1-r_n)^{\sigma/2}}{(1-rr_n)^{2\sigma-\gamma-2}}\\
&\leq&
c\left(\sum_{n=1}^N |\beta_n|^2 \right)^{1/p'} \left(\sum_{n=1}^N
\frac {(1-r_n)^{\sigma
q'/2}}{(1-rr_n)^{(2\sigma-\gamma-2)q'}}\right)^{1/q'}\\
&\leq& c' \, \|f\|_2^\sigma  \left(\sum_{n=1}^N \frac
{(1-r_n)^{\sigma
q'/2}}{(1-rr_n)^{(2\sigma-\gamma-2)q'}}\right)^{1/q'}\\
&\leq& \frac {c'' \,  \|f\|_2^\sigma} {\big(
(1-r)^{-(1-p)-(1+\gamma)(1+p/2)} \, h(1-r) \big)^{1/q'}}.
\end{eqnarray*}
\end{proof}

\noindent Now, we can apply Theorem \ref{T:modberg} for the special
function defined in (\ref{E:defh}).

{\bf Case I:} If $z_n=r_n e^{i\theta_n}$, $n\geq 1$, is a Carleson
sequence satisfying
\[
\sum_{n=1}^\infty (1-r_n)^{\alpha} \, (\log \frac{1}{1-r_n}
)^{\alpha_1}\cdots (\log_m \frac{1}{1-r_n} )^{\alpha_m} <\infty,
\]
with $p\in (2/3,1)$, $\sigma=2p/(p+2)$, $-1 < \gamma < 2\sigma-2$,
and $(1-p)+(1+\gamma)(1+p/2) < \alpha < p/2$, then, for all $f\in
K_B$, we have
\begin{eqnarray*}
& & \left( \int_{0}^{1}\int_{0}^{2\pi} \, \left|f'(r \, \rho
e^{i\theta})\right|^\sigma\, \, \rho(1-\rho^2)^\gamma d\rho \,
d\theta \right)^{1/\sigma}\\
&\leq&
\frac {C \, \|f\|_2}{\big(\, (1-r)^{\alpha-(1-p)-(1+\gamma)(1+p/2)}
\,\, (\log \frac{1}{1-r} )^{\alpha_1}\cdots (\log_m \frac{1}{1-r}
)^{\alpha_m} \,\big)^{1/p}}
\end{eqnarray*}
as $r\to 1^-$.

{\bf Case II:} If $z_n=r_n e^{i\theta_n}$, $n\geq 1$, is a Carleson
sequence satisfying
\[
\sum_{n=1}^\infty (1-r_n)^{(1-p)+(1+\gamma)(1+p/2)} \, (\log_k
\frac{1}{1-r_n} )^{\alpha_k}\cdots (\log_m \frac{1}{1-r_n}
)^{\alpha_m} <\infty,
\]
with $p\in (2/3,1)$, $\sigma=2p/(p+2)$, $-1 < \gamma < 2\sigma-2$,
  $\alpha_k,\alpha_{k+1},\dots,\alpha_m \in \R$, and $\alpha_k <
0$, then, for all $f\in K_B$, we have
\[
\left( \int_{0}^{1}\int_{0}^{2\pi} \, \left|f'(r \, \rho
e^{i\theta})\right|^\sigma\, \, \rho(1-\rho^2)^\gamma d\rho \,
d\theta \right)^{1/\sigma}
\leq
\frac {C \, \|f\|_2}{\big(\, (\log_k \frac{1}{1-r}
)^{\alpha_k}\cdots (\log_m \frac{1}{1-r} )^{\alpha_m} \,\big)^{1/p}}
\]
as $r\to 1^-$. However, if
\[
\sum_{n=1}^\infty (1-r_n)^{(1-p)+(1+\gamma)(1+p/2)}<\infty,
\]
then,  we still have
\[
\left( \int_{0}^{1}\int_{0}^{2\pi} \, \left|f'(r \, \rho
e^{i\theta})\right|^\sigma\, \, \rho(1-\rho^2)^\gamma d\rho \,
d\theta \right)^{1/\sigma}
\leq
C \, \|f\|_2,
\]
which means that
\[
f' \in A_\gamma^{2p/(p+2)},
\]
and the differential operator
\[
\begin{array}{ccl}
K_B & \longrightarrow & A_\gamma^{2p/(p+2)},\\
f & \longmapsto & f',
\end{array}
\]
is bounded.

{\bf Case III:} If $z_n=r_n e^{i\theta_n}$, $n\geq 1$, is a Carleson
sequence satisfying
\[
\sum_{n=1}^\infty (1-r_n)^{p/2} \, (\log_k \frac{1}{1-r_n}
)^{\alpha_k}\cdots (\log_m \frac{1}{1-r_n} )^{\alpha_m} <\infty,
\]
with $p\in (2/3,1)$, $\sigma=2p/(p+2)$, $-1 < \gamma < 2\sigma-2$,
 $\alpha_k,\alpha_{k+1},\dots,\alpha_m \in \R$ and $\alpha_k >
0$, then, for all $f\in K_B$, we have
\begin{eqnarray*}
& & \left( \int_{0}^{1}\int_{0}^{2\pi} \, \left|f'(r \, \rho
e^{i\theta})\right|^\sigma\, \, \rho(1-\rho^2)^\gamma d\rho \,
d\theta \right)^{1/\sigma}\\
&\leq&
\frac {C \, \|f\|_2}{\big(\, (1-r)^{(3p/2-1)-(1+\gamma)(1+p/2)} \,\,
(\log_k \frac{1}{1-r} )^{\alpha_k}\cdots (\log_m \frac{1}{1-r}
)^{\alpha_m} \,\big)^{1/p}}
\end{eqnarray*}
as $r\to 1^-$. However, if
\[
\sum_{n=1}^\infty (1-r_n)^{p/2} <\infty,
\]
then,  we still have
\[
\left( \int_{0}^{1}\int_{0}^{2\pi} \, \left|f'(r \, \rho
e^{i\theta})\right|^\sigma\, \, \rho(1-\rho^2)^\gamma d\rho \,
d\theta \right)^{1/\sigma}
\leq
\frac {C \, \|f\|_2}{\,\,\,
(1-r)^{\big(\,(3p/2-1)-(1+\gamma)(1+p/2)\,\big)/p} }
\]
as $r\to 1^-$.

\end{document}